\documentclass[12pt,a4paper]{amsart}
\usepackage{graphicx}
\usepackage{float,epsfig}
\usepackage[russian,english]{babel}
\usepackage[psamsfonts]{amssymb}
\usepackage{amsfonts}
\usepackage{amsmath}
\newtheorem{theor}{Theorem}
\newtheorem{lemma}{Lemma}
\newtheorem{cor}{Corollary}
\newtheorem{remark}{Remark}
\def\Re{\mathop{\rm Re}\nolimits}
\def\Im{\mathop{\rm Im}\nolimits}
\def\th{\mathop{\rm th}\nolimits}
\def\sn{\mathop{\rm sn}\nolimits}
\def\cn{\mathop{\rm cn}\nolimits}
\def\dn{\mathop{\rm dn}\nolimits}

\begin{document}

\title[]{COMPARISON OF HYPERBOLIC METRIC AND\\ TRIANGULAR RATIO METRIC IN A SQUARE}\thanks{The work of the second author was performed under the development program of the Volga Region
Mathematical Center (agreement no. 075-02-2025-1725/1).}

\author[A.~Kushaeva]{A.~Kushaeva}
\address{Kazan Federal University,
         Kazan, Russia}
\email[]{kushaeva1710@mail.ru}

\author[S.~Nasyrov]{S.~Nasyrov}
\address{Kazan Federal University,
         Kazan, Russia}
 \email[]{semen.nasyrov@yandex.ru}

\maketitle

\begin{abstract}
        Let $K$ be a square in the plane and $\rho_K(x,y)$ be the hyperbolic distance between $x$, $y\in K$.
        Denote by $s_K(x,y)$ the triangular ratio metric in $K$; for $x\neq y$ the value of $s_K(x,y)$ equals the ratio of the
        Euclidean distance $|x-y|$ between $x$, $y\in K$ to the value $\inf_{z\in \partial K}(|x-z|+|z-y|)$.
        We obtain a sharp estimate for the ratio of  $\th (\rho_K(x,y)/2)$ to  $s_K(x,y)$.

\noindent {Keywords:} {hyperbolic metric, triangular ratio metric, square}.

\noindent {Mathematics Subject Classification:} 51M09; 51M16; 30C20.
\end{abstract}

\section*{Introduction}
The hyperbolic metric plays an important role in the geometric function theory.
One of the main properties of the hyperbolic metric  is its conformal invariance:
the hyperbolic distance keeps under conformal mapping of domains of hyperbolic type.
Also the hyperbolic metric does not increase under holomorphic mappings what is a corollary of the
Schwarz lemma. In a simply connected domain $D$ of hyperbolic type,
the hyperbolic metric can be easily calculated if we know the
conformal mapping of the domain onto a canonic domain such as the
unit disc or the half-plane. If the mapping has a simple form, then
we can  calculate the hyperbolic distance between points
sufficiently well. Otherwise, the problem is to estimate  the
hyperbolic metric.

To do this, we can try to find another metric in  $D$,  equivalent
in some sense to the hyperbolic one, which can be simply calculated
in Euclidean terms. Many such metrics are known (see, e.g., the
monographs \cite{HKV,papa}). One of such metrics is the triangular
ratio
 metric $s=s_D$. Is is defined in $D$ as follows: if
$x\neq y$, then
$$
s_D(x,y)=\frac{|x-y|}{\inf_{z\in \partial D}(|x-z|+|z-y|)},
$$
and if $x=y$, then $s_D(x,y)=0$. Here $|x-y|$ denotes the Euclidean distance between $x$ and $y$.

The properties of this metric for some classes of domains and its
applications in the theory of conformal and quasiconformal mappings
were studied in many works (see, e.g.,
\cite{dknv,ranio,ranio2,ranio_vuor2,ranio_vuor}).

Very important planar domains are polygonal ones, in particular rectangles, see, e.g., \cite{b1,b2,dknv}. In this paper, we investigate the hyperbolic metric in a square.

In \cite{dknv}, the problem of finding the best constants in the
inequality
\begin{equation}\label{1d}
C_1\,s_D(x,y)\le\th (\rho_D(x,y)/2)\le C_2\,s_D(x,y), \quad x,y \in
D,
\end{equation}
 is investigated for some convex polygonal planar domains. It is easy to show that
  the maximal constant $C_1$ in the inequality is equal to $1$. The equality is attained in the
  limit when one of the points tends to the boundary of $D$. The problem of finding the
  minimal constant $C_2$ is more complicated. In \cite{dknv} a more simple problem to estimate the value
\begin{equation*}
\lim_{y\to x}\frac{\th (\rho_D(x,y)/2)}{s_D(x,y)}
\end{equation*}
for all $x\in D$ was solved. It is easy to see that
$$
\lim_{y\to x}\frac{\th (\rho_D(x,y)/2)}{s_D(x,y)}= \frac{2d_D(x)}{r_D(x)},
$$
where $d_D(x)=\mbox{\rm dist}(x,\partial D)$ is the distance from
$x$ to the boundary $\partial D$ of $D$, and $r_D(x)$ is the
conformal radius of $D$ in the point $x$, i.e., the radius of disc
$\{|z|<r=r_D(x)\}$, onto which we can map conformally the domain $D$
by the function $f$ with normalization $f(x)=0$, $f'(x)=1$.

In \cite{dknv}, the sharp estimates of the value
\begin{equation}\label{2d}
\frac{2d_D(x)}{r_D(x)},
\end{equation}
were established for many often used convex polygonal domains $D$,
for example, rectangles, isosceles triangles, and regular $n$-gons.
Also the set, on which the maximum of the value \eqref{2d} can be
attained, was described there.

Now we will describe some results from \cite{dknv} obtained for rectangles of the form $R=[-k,k] \times[-1,1]$, $k\ge 1$.  As it was shown in \cite[thrm.~4.7]{dknv},
the maximum of the ratio
$
{2d_R(x)}/{r_R(x)}
$
is attained in the point of intersection of the bisectors of adjacent angles of $R$; it is equal to
$$C(\lambda)=
\mathcal{K}(\lambda)\left|\frac{\cn(i\mathcal{K}(\lambda),\lambda)\dn(i\mathcal{K}(\lambda),
\lambda)}{\sn(i\mathcal{K}(\lambda),\lambda)}\right|.
$$
Here $\mathcal{K}(\lambda)$ is the complete elliptic integral of the first kind, the functions  $\sn$, $\cn,$  and $\dn$ are the Jacobi elliptic functions (see, e.g., \cite{akhiezer,avv}), and the parameter $\lambda$ is found from the relation
$$
\frac{2\mathcal{K}(\lambda)}{\mathcal{K}(\lambda')}=k, \quad
\lambda'=\sqrt{1-\lambda^2}.
$$

In particular, in the case of the square $K=[-1,1]\times[-1,1]$ the maximum of
$
{2d_K(x)}/{r_K(x)}
$
is attained at its center and it  is equal to $C(\lambda_0),$ where $\lambda_0=3-2\sqrt{2}$. It is easy to verify (see, e.g., \cite[thrm.~1.2]{b1}) that
\begin{equation}\label{3d}
C(\lambda_0)=\mathcal{K}(\sqrt{2}/2)=1.854074677\ldots
\end{equation}

As for the finding the best constant $C_2$ in \eqref{1d}, the
problem was open even for the case of a square. In
\cite[conj.~4.11]{dknv}, it was suggested that for arbitrary points
from $R$ the sharp estimate
$$
{s_R(x,y)}\le{\th\frac{\rho_R(x,y)}{2}}  \le
C(\lambda){s_R(x,y)}\,$$ holds. In this paper, we show that the
conjecture is valid for the case when $R$ is the square $K=[-1,1]^2$
and prove the following statement.

\begin{theor}\label{square}
For arbitrary points $x$ and $y$ from $K$ we have the sharp inequality
$$
{s_K(x,y)}\le{\th\frac{\rho_K(x,y)}{2}}  \le
C(\lambda_0){s_K(x,y)},$$ where $C(\lambda_0)$ is defined by
\eqref{3d}.
\end{theor}

Now we will describe the structure of the paper.
In \S~1, we establish some properties of pairs $(x,y)\in K,$ where
the maximum of the ratio
$$\frac{\th (\rho_K(x,y)/2)}{s_K(x,y)}$$
is attained.
In particular, we prove that, in the case,
$x$ and $y$ are on the opposite sides of a sufficiently small square
homothetic to $K$ and satisfy some additional requirement
(Lemma~\ref{t4}).  Then, in \S~2 we give a proof of
Theorem~\ref{square}.

\section{Some auxiliary properties of the ratio of $\mbox{\rm th}({\rho}/{2})$ and $s$-metric in~the~square}

Our main problem is to estimate the value of
$\displaystyle
\frac{\th(\rho_K(x,y)/2)}{s_K(x,y)}$, $x\neq y$, where  $x$ and $y$
are in the square $K=ABCD$ (Fig.~1).
The diagonals of the square $ABCD$ subdivide it into four triangles. Because of the symmetry of the square, we may assume that the point $x$ is in the triangle $AOD$,
adjacent to the lower side $AD$ of the square. If the point $x$ is fixed, then the infimum (minimum)
\begin{equation}\label{4d}
\inf_{z\in \partial K}(|x-z|+|y-z|)
\end{equation}
is attained on one of the sides of the square. We will find out how the location of the second point $y$ affect on the location of the point $z$ at which the minimum is attained. Note that in \cite[\S~3]{dknv} such situation was described for the case of arbitrary rectangle. Here we will consider the case of the square in more details. Consider the points $p$ and
$q$ in the square $ABCD$ satisfying the following properties: \smallskip

\noindent (1) the segments $Ap$ and $Ax$  lie on the straight lines,
symmetric to each other with respect to the diagonal $AC$ of the
square $ABCD$; \smallskip

\noindent (2) the segments $Dq$ and $Dx$  lie on the straight lines,
symmetric to each other with respect to the diagonal $BD$ of the
square $ABCD$; \smallskip

\noindent (3) $\Im p=\Im q=- \Im x$. \smallskip

\noindent Note that the points $p$ and $q$ may coincide if $x$ is the center of the square.

The segments $\Sigma_1=Ap$, $\Sigma_2=Bp$,  $\Sigma_3=Cq$, $\Sigma_4=Dq$ and $\Sigma_5=pq$ separate
 the square into four parts; we will denote them by $G_1$, $G_2$, $G_3,$ and $G_4$ (Fig.~1).
Further, if we want to point out that the segments $\Sigma_k$
correspond to the point $x$, we will write $\Sigma_k^x$ instead of
$\Sigma_k$.

%%%%%%%%%%%%%%%%%%%%%%%%%%%%%%%%%%%%%%%%%%%%%%%%%%%%%%%%%%%%%%%%%%%%%%%%%%%%%%%%%%%%%%%%%%%%%%%%%%%%%%%%%%%%%%%%%%%%%%%%%%%%%%%%%%

\begin{figure}[h]
\begin{center}
\includegraphics[width=0.3\textwidth]{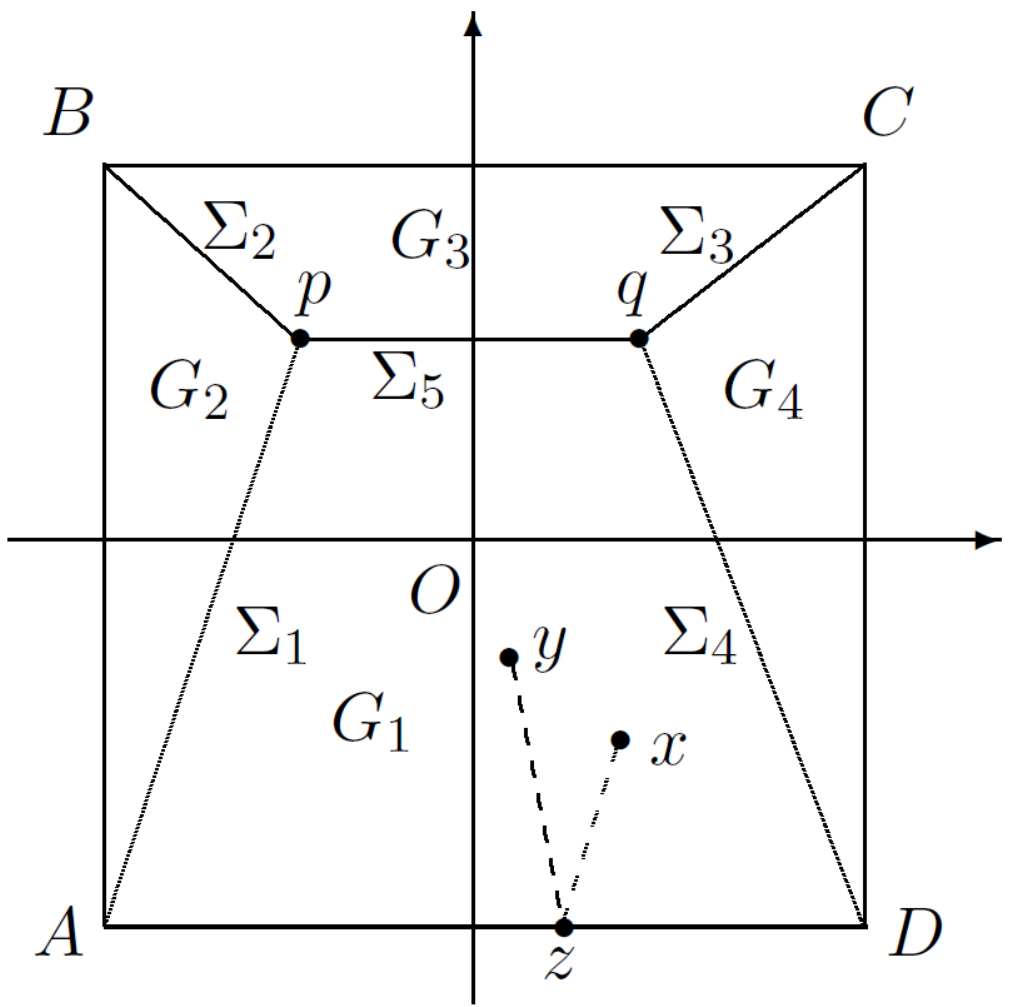}
\caption{Decomposition of a square into parts $G_j$ depending on
the location of the point $z$ at which the infimum
\eqref{4d} is attained, under  the fixed point $x$.} \label{fig1}
\end{center}
\end{figure}

%%%%%%%%%%%%%%%%%%%%%%%%%%%%%%%%%%%%%%%%%%%%%%%%%%%%%%%%%%%%%%%%%%%%%%%%%%%%%%%%%%%%%%%%%%%%%%%%%%%%%%%%%%%%%%%%%%%%%%%%%%%%%%%%%%%%%%

\begin{lemma}\label{l1}
        Let $x$ be in the triangle $AOD$. If $y\in G_1,$ then the infimum \eqref{4d} is attained at a point $z\in AD$ and the triangular ratio metric between $x$ and $y$ equals
\begin{equation}\label{5d}
\displaystyle s_K(x,y) = \frac{|x-y|}{|x-\overline{y}+2i|}.
\end{equation}
If $y\in G_2,$ then the infimum \eqref{4d} is attained at  a point $z\in AB$ and
\begin{equation}\label{6d}
\displaystyle s_K(x,y) = \frac{|x-y|}{|x+\overline{y}+2|}.
\end{equation}
If $y\in G_3,$ then the infimum \eqref{4d}  is attained at  a point $z\in BC$ and
\begin{equation}\label{7d}
\displaystyle s_K(x,y) = \frac{|x-y|}{|x-\overline{y}-2i|}.
\end{equation}
At last, if $y\in G_4,$ then the infimum \eqref{4d}  is attained at  a point $z\in CD$ and
\begin{equation}\label{8d}
\displaystyle s_K(x,y) = \frac{|x-y|}{|x+\overline{y}-2|}.
\end{equation}
\end{lemma}

\begin{proof}
Let the point $z$, at which the infimum \eqref{4d} is attained, be
on the side $AD$, then consider the point $y^*$, symmetric to $y$
with respect to the straight line containing $AD$. It is obvious
that $|x-z|+|z-y| = |x-z|+|z-y^*|$, from which it follows
that the infimum of $|x-z|+|z-y|$ is attained at the point $z \in
\partial K$ such that  $z$ is on the segment with endpoints $x$  and $y^{*}$, consequently,
$$
|x-z|+|z-y| = |x-z|+|z-y^*|=|x-y^*|.
$$
Since $y^* = \overline{y} - 2i$, we have
$$
\displaystyle s_R(x,y) = \frac{|x-y|}{|x-y^*|}=
\frac{|x-y|}{|x-\overline{y}+2i|},$$ and from this we obtain
\eqref{5d}. Similarly we prove that in the case when the point $z$,
at which the infimum \eqref{4d} is attained, is on the sides $AB$,
$BC$ and $AD$,
 we have equalities \eqref{6d}, \eqref{7d}, and \eqref{8d},
 and we take  $y^*$ symmetric to $y$ with respect to the straight lines on which the corresponding sides of the square lie.

Now we will study the case when the infimum \eqref{4d} is attained
at two different points lying on distinct sides of the square. First
assume that these are adjacent sides, say, $AD$ and $AB$. Then,
equating the right-hand sides of \eqref{5d} and \eqref{6d}, we
obtain $$|x-\overline{y}+2i|=|x+\overline{y}+2|.$$ The vectors
$u=Ax$ and $v=Ay$ have the forms $u=x+1+i$ and $v=y+1+i$ and satisfy
$|u-\overline{v}|=|u+\overline{v}|$. From this it follows that
$\Re(uv)=0$ which means that $x$ and  $y$ are symmetric to each
other with respect to the straight line containing the diagonal $AC$
of the square. If the infimum is attained on the opposite sides,
say, on $AD$ and $AB$, then
$|x-\overline{y}+2i|=|x-\overline{y}-2i|$, and this means
that $\Im y=-\Im x$. Therefore, the infimum in \eqref{4d} is
attained at two different points iff  $y$ is on one of the segments
$\Sigma_j$, $1\le j\le 5$.  The statement of Lemma~\ref{l1} is
proved.
\end{proof}

\medskip

Now we will describe the points $p$ and $q$ in terms of  $x$.

\begin{lemma}\label{pq}
Let $x$ be in the triangle $AOD$ and  $t=\Re x$, $a=-\Im x$, $a\in
[0,1)$. Then,
$$
\Re p= -\frac{a^2+t}{1+t}, \quad \Re q= \frac{a^2-t}{1-t},\quad  \Im p=\Im q=a.
$$
\end{lemma}

\begin{proof}
We will only prove the formulas for the point $p$, for $q$ the proof
is quite similar. Denote $b=\Re p$.  From the proof of
Lemma~\ref{l1} it follows that for the points $u=x+1+i$ and
$v=p+1+i$  the equality $\Re (uv)=0$ holds. We have
$$
\Re(uv)=(t-ia+1+i)(b+ia+1+i)=(1+t)(1+b)-(1-a^2).
$$
Consequently,
$$b=\frac{1-a^2}{1+t}-1=-\frac{a^2+t}{1+t}.$$
\end{proof}

\begin{cor}\label{cr1}
If $b=\Re p$, $c=\Re q$, then $(1+t)(1+b)=1-a^2$, $(1-t)(1-c)=1-a^2$.
\end{cor}

\begin{lemma}\label{t3} Let $x$ from the triangle $AOD$ be fixed. Then,
the maximum of the value
\begin{equation}\label{9d}
\displaystyle\frac{\th ({\rho_K(x,y)}/{2})}{s_K(x,y)}
\end{equation}
is attained at a point  $y$ lying on one of the segments $\Sigma_k$,
$1\le k\le 5$.
    \end{lemma}
    \begin{proof} To prove the lemma, we will use the maximum principle for analytic functions.

Consider the case when $y$ is in $G_1$ (for other domains $G_j$ the
proof is similar). Let $f$ conformally map the square $K$ onto the
upper half-plane. Introduce the function
$$
g(y) = \displaystyle\frac{\th \frac{\rho_K(x,y)}{2}}{s_K(x,y)}
$$
By \eqref{5d},
$$
g(y) = \frac{|f(x)-f(y)|\,|x-\overline{y}+2i|}{|f(x)-\overline{f(y)}|\,|x-y|}.
$$
It is easy to show that $g$ is the modulus of the function
$$
\frac{f(x)-f(y)}{x-y}\,\frac{y-\overline{x}-2i}{f(y)-\overline{f(x)}},
$$
holomorphic in~$G_1$.

By the maximum principle, the maximum of $g(y)$ is attained either
on one of the segments $\Sigma_k$, $3\le k\le 5$, or on the lower
side of the square. But if $y$ lies on the  lower side of the
square, then $g(y)=1$. Since   $g$ satisfies $g(y)\ge 1$, $y\in
G_1$, we see that the maximum can be only attained at one of the
segments $\Sigma_k$, $3\le k\le 5$.
 \end{proof}

The following lemma clarifies the statement of Lemma~\ref{t3}.

    \begin{lemma}\label{t4} The maximum of the value \eqref{9d}
is attained at a point $(x,y),$ where $x$  is in the triangle $AOD$ and $y$ is on the segment~$\Sigma_5^x$.
    \end{lemma}

    \begin{proof}
Let the maximum of \eqref{9d} be attained at some point $(x,y)$.
Then, $x$ and $y$ are not on the boundary of $K$, since at such
points the value of \eqref{9d} equals $1$.   We can assume that $x$
is in the triangle $AOD$. Then, by Lemma~\ref{t3}, the point  $y$ is
on one of the segments $\Sigma_k$. We will show that if $y\not\in
\Sigma_5$, then there exists a point $(\tilde{x},\tilde{y})$ from
the square $ABCD$ at which the value of \eqref{9d} is strictly
bigger that that at the point $(x,y)$.

Let, for certainty, $y \in \Sigma_1$ and $y\neq p$.
Then, the minimum $m=\min_{z\in \partial K}(|y-z|+|z-y|)$ is attained
at two different points lying on the sides $AB$ and $AD$ of the
square $K$, i.e.,
$$
\min_{z\in AB}(|x-z|+|z-y|)=\min_{z\in  AD}(|x-z|+|z-y|)=m,
$$
and, in addition,  $$ M=\min_{z\in BC\cup CD}(|x-z|+|z-y|)>m.$$

Now consider the homothety $\phi$ with center at the point $A$,
lower left vertex of the square, and ratio $\alpha=1+\varepsilon$,
$\varepsilon>0$. It has the form $w=\phi(z)=\alpha (z+(1+i))-(1+i)$.
Denote $\tilde{x}=\phi(x)$, $\tilde{y}=\phi(x)$. We will assume that
$\varepsilon$ is so small that the points $\tilde{x}$ and
$\tilde{y}$ are inside  the square $K$. We have
$|\tilde{x}-\tilde{y}|=\alpha|x-y|$. For sufficiently small
$\varepsilon$
$$
\min_{z\in AB}(|\tilde{x}-z|+|z-\tilde{y}|)=\min_{z\in AD}(|\tilde{x}-z|+|z-\tilde{y}|)=\alpha m.
$$
Since $M>m$ and the value $\min_{z\in BC\cup
CD}(|x-z|+|z-y|)$ depends continuously on the points $x$ and $y$, for small $\varepsilon,$ we have
$$
\min_{z\in BC\cup CD}(|\tilde{x}-z|+|z-\tilde{y}|)>\alpha m.
$$
Therefore,
$$
\min_{z\in \partial K}(|\tilde{x}-z|+|z-\tilde{y}|)=\alpha m
$$
and then
\begin{equation}\label{sk}
 s_K(\tilde{x},\tilde{y})= s_K(x,y).
\end{equation}

Now let us consider the behavior of the hyperbolic metric
$\displaystyle\th \frac{\rho_K(x,y)}{2}$ under the homothety $\phi$,
which is a conformal mapping. Applying $\phi$ to the square $K$,
because of the invariance of the hyperbolic metric under conformal
mappings,  we obtain  $\rho_K(x,y) =
\rho_{\tilde{K}}(\tilde{x},\tilde{y}),$ where $\tilde{K}$ is the
image of $K$ under the mapping $\phi$. It is obvious that $K \subset
\tilde{K}$ and   $K \neq\tilde{K}$.

Since the hyperbolic metric increases under contraction of domain
(see, e.g., \cite[ch.~VIII]{goluzin}),
we have
$$
\rho_K(\tilde{x},\tilde{y}) >  \rho_{\tilde K}(\tilde{x},\tilde{y}) =  \rho_R(x, y).
$$
Therefore,
$$
\displaystyle\th \frac{\rho_K(\tilde{x},\tilde{y})}{2} \geq  \th
\frac{\rho_K(x, y)}{2},
$$
and from this, taking into account \eqref{sk}, we obtain
$$\displaystyle\frac{\th \frac{\rho_R(\tilde{x},\tilde{y})}{2}}{s_R(\tilde{x},\tilde{y})}>\displaystyle\frac{\th \frac{\rho_R(x,y)}{2}}{s_R(x,y)},
$$
which proves the lemma.
\end{proof}

\section{Proof of Theorem~1}

To prove Theorem~\ref{square}, we need to show that for every points $w_1$, $w_2\in K$ we have
\begin{equation}\label{srho}
{\th\frac{\rho_K(w_1,w_2)}{2}}  \le C(\lambda_0){s_K(w_1,w_2)}\,.
\end{equation}
By Lemma~\ref{t4}, we can assume that
\begin{equation}\label{w12}
w_1=u_1-ia, \ w_2=u_2+ia, \quad \text{where} \quad  |u_1|, |u_2|\le a,
\quad 0<a<1,
\end{equation}
i.e., the points $w_1$ and $w_2$ are on the lower and upper sides of the square centered at the origin with length of side $2a$.
Because of Corollary~\ref{cr1} it is sufficient to consider the case when $u_1$ and $u_2$ satisfy the inequalities
$(1+u_1)(1+u_2)\ge 1-a^2$, $(1-u_1)(1-u_2)\ge 1-a^2$, which can be written in the form
\begin{equation}\label{13d}
|u_1+u_2|\le a^2+u_1u_2.
\end{equation}

The conformal mapping of the unit disc $U=\{|z|<1\}$ onto the square
$K=[-1,1]^2$ has the form
\begin{equation}\label{10d}
f(z)=C^{-1}\int_0^z \frac{d\zeta}{\sqrt{1+\zeta^4}},
\end{equation}
where
$$
C=\int_0^1 \frac{dt}{\sqrt{1+t^4}}=\frac{C(\lambda_0)}{2}=0.927037...
$$
Then, the preimages $z_1$, $z_2$  of the points $w_1$, $w_2$ under the mapping $f$ are in the disc $|z|\le r=r(a),$ where $re^{i\pi/4}$ is the preimage of the vertex $a\sqrt{2}e^{i\pi/4}$ of the square. Consequently, $a$ and $r=r(a)$ are connected via relation
\begin{equation}\label{ar}
\sqrt{2} Ca=B(r), \quad \text{where}\quad B(r)=\int_0^{r} \frac{dt}{\sqrt{1-t^4}}.
\end{equation}

\begin{lemma} \label{t6}
Let $w_1$ and $w_2$ have the form \eqref{w12}, $u_1$ and $u_2$
satisfy \eqref{13d} and $r=r(a)$ be defined by \eqref{ar}. Then,
\begin{equation}\label{frac}
\frac{\th\frac{\rho_K(w_1,w_2)}{2}}{s_K(w_1,w_2)}\le 2C F(r),
\end{equation}
where
$$
F(r)=\frac{\sqrt{2}r}{(1+r^2)B(r)}\,.
$$
\end{lemma}

\begin{proof} By the hyperbolic metric principle (see, e.g., \cite[ch.~VIII]{goluzin}) we have
$\rho_K(w_1,w_2)=\rho_U(z_1,z_2),$ where $|z_k|\le r$, $k=1$, $2$.
The non-Euclidean diameter of the disc $\{|z|\le r\}$  equals the
distance between the diametrically opposite points $r$ and $-r$,
therefore,
$$
\th\frac{\rho_K(w_1,w_2)}{2}\le\th\frac{\rho_U(-r,r)}{2}=\frac{2r}{1+r^2}\,.
$$
The $s$-metric between $w_1$ and  $w_2$  in the square $K$ satisfies
\begin{equation}\label{sK}
s_K(w_1,w_2)=\sqrt{\frac{(u_1-u_2)^2+4a^2}{(u_1-u_2)^2+4}}\ge a.
\end{equation}
These inequalities imply \eqref{frac}.
\end{proof}

Let $r_0=0.625623$ and
\begin{equation}\label{11d}
a_0=\frac{B(r_0)}{\sqrt{2}C}=0.485087...
\end{equation}

\begin{cor} \label{t7}
Let $w_1$  $w_2$ satisfy the conditions of Lemma~\ref{t6}, where $a\ge a_0$. Then,  inequality \eqref{srho} holds.
\end{cor}

\begin{proof}
It is easy to verify that the function $F$ strictly increases on $(0,1)$ and for $r=r_0=0.625623$ we have $F(r_0)<1$.
Then, for $r>r_0$ the inequality $F(r)<1$ holds and, therefore, for $a\ge a_0$
\begin{equation}\label{2C}
\frac{\th\frac{\rho(w_1,w_2)}{2}}{s(w_1,w_2)}
\le {2}{C}\,,
\end{equation}
what follows \eqref{srho}.
\end{proof}

Because of Corollary~\ref{t7} it remains to prove \eqref{2C} for points satisfying \eqref{w12} and \eqref{13d},
for $0<a<a_0$.

\begin{lemma} \label{t8}
Let
$w_1$ and $w_2$ satisfy \eqref{w12} and \eqref{13d}, $a<a_0$ and $r=r(a)$. Then,
$$
\th\frac{\rho(w_1,w_2)}{2}
\le C\frac{(1+A(r))|w_1-w_2|}{ |1-C^2w_1\overline{w_2}|-2A(r)B^2(r)-A^2(r)B^2(r)},
$$
where
$
A(r)=\sqrt{1+r^4}-1.
$

\end{lemma}

\begin{proof}
Since the hyperbolic metric is invariant under conformal mappings, we have
$$
\th\frac{\rho_K(w_1,w_2)}{2}=\th\frac{\rho_U(z_1,z_2)}{2}=\frac{|z_1-z_2|}{|1-z_1\overline{z_2}|},
$$
where  $z_1$ and $z_2$ are the preimages of $w_1$ and $w_2$ under the mapping $f$. Denote $g=f^{-1}$. Then,
$$
|z_1-z_2|=\left|\int_{w_1}^{w_2}g'(w)dw\right|.
$$
In addition, $g'(w)=C \sqrt{1+z^4}$, where $z=g(w)$, and
$$
|g'(w)-C|=C |\sqrt{1+z^4}-1|
\le CA(r).
$$
Then,
$$
|g'(w)|\le C (1+A(r)).
$$
Thus,
$$
|z_1-z_2|\le C (1+A(r))|w_1-w_2|.
$$
Now we will estimate  the value $|1-z_1\overline{z_2}|$ from below.
We have
$$
|z_k-Cw_k|=\left| \int_0^{w_k}(g'(w)-C)dw\right|\le C A(r)|w_k|\le A(r)B(r).
$$

Since, by \eqref{10d},
$$
|w_k|=|f(z_k)|=C^{-1}\left|\int_0^{z_k}\frac{d\zeta}{\sqrt{1+\zeta^4}}\right|\le C^{-1}B(|z_k|)\le C^{-1}B(r),
$$
with the help of the triangle inequality we obtain
\begin{multline*}
|1-z_1\overline{z_2}|=|1-[(z_1-Cw_1)+Cw_1][(\overline{z_2}-C\overline{w_2})+C\overline{w_2}]|
\\
\ge |1-C^2w_1\overline{w_2}|-2A(r)B^2(r)-A^2(r)B^2(r).
\end{multline*}
Therefore,
$$
\frac{|z_1-z_2|}{|1-z_1\overline{z_2}|}\le C\frac{(1+A(r))|w_1-w_2|}{ |1-C^2w_1\overline{w_2}|-2A(r)B^2(r)-A^2(r)B^2(r)},
$$
what follow the statement of the lemma.
\end{proof}

\begin{remark}\label{rem1}
Nonnegativity of the expression $|1-C^2w_1\overline{w_2}|-2A(r)B^2(r)-A^2(r)B^2(r)$
 in the denominator of the fraction above for $a<a_0$ follows from the estimates given below.
\end{remark}

From the estimate obtained in the lemma, with account of the
explicit form of  the $s$-metric (see \eqref{sK}), if follows that
for the proof of \eqref{2C} for $a<a_0$ it is sufficient to show
that \begin{multline} \label{12d} (1+A(r))\sqrt{1+(u_1-u_2)^2/4}\le
\\
\sqrt{(1+C^2(a^2-u_1u_2))^2+C^4a^2(u_1+u_2)^2}-2A(r)B^2(r)-A^2(r)B^2(r),
\end{multline}
where $u_1$ and $u_2$ satisfy \eqref{13d}, and the values $a$ and $r$ are connected by \eqref{ar}.

Denote $\Phi(r)=A(r)+2A(r)B^2(r)+A^2(r)B^2(r)$.

\begin{lemma} \label{t9}
Assume that $u_1$ and $u_2$ satisfy \eqref{13d} for some $a<a_0$, and for some $\alpha$, $\beta>0$ the inequalities
\begin{equation}\label{abc}
(u_1-u_2)^2\le \alpha (a^2-u_1u_2), \quad a^2-u_1u_2\ge \beta a^2
\end{equation}
hold. If
\begin{equation}\label{16d}
\Phi(r)\le(C^2-\alpha/8(1+A(r_0)))(a^2-u_1u_2), \quad \text{where}\quad r=r(a),
\end{equation}
then \eqref{12d} is valid.
\end{lemma}

\begin{proof} From \eqref{16d} with the help of \eqref{abc} and monotonicity of the function $A(r)$
it follows that
\begin{equation*}\label{15d}
(1+A(r))(u_1-u_2)^2/8+\Phi(r)\le C^2(a^2-u_1u_2),
\end{equation*}
which can be written in the form
\begin{equation*}
(1+A(r))(1+(u_1-u_2)^2/8)+2A(r)B^2(r)+A^2(r)B^2(r)\le 1+C^2(a^2-u_1u_2).
\end{equation*}
From the latter inequality taking into account that   $\sqrt{1+(u_1-u_2)^2/4}\le 1+(u_1-u_2)^2/8$ we obtain \eqref{12d}.
\end{proof}

\begin{lemma} \label{t10}
Assume that $u_1$ and $u_2$ satisfy \eqref{13d} for some $a<a_0$ and for some $\alpha$, $\beta>0$ inequalities \eqref{abc} hold.
If
\begin{equation}\label{19d}
\frac{A(r)}{B^2(r)}+2A(r)+A^2(r)\le \gamma,
\end{equation}
where $r=r(a)$ and
\begin{equation*}
\gamma=\frac{\beta(C^2-\alpha/8(1+A(r_0)))}{2C^2},
\end{equation*}
then  \eqref{16d} is valid.
\end{lemma}

\begin{proof}
From \eqref{19d}, taking into account \eqref{ar}, we obtain
\begin{equation*}
\Phi(r)\le\gamma B(r)^2=2\gamma C^2 a^2=\beta(C^2-\alpha/8(1+A(r_0)))a^2 ,
\end{equation*}
which, with the help of the second inequality from  \eqref{abc},
follows \eqref{16d}.
\end{proof}

Therefore, it remains to show that for $r<r_0$ inequality
\eqref{19d} holds. It is easy to show that the expression in the
left-hand side of \eqref{19d} strictly increases as a function of
$r$ on the segment $[0,r_0]$, and
\begin{equation*}
\frac{A(r_0)}{B^2(r_0)}+2A(r_0)+A^2(r_0)=
0.314881.
\end{equation*}
Therefore, for the proof of Theorem~\ref{square} we only have to establish that
\begin{equation}\label{20d}
\gamma>0.314881.
\end{equation}

To estimate $\gamma$, we will prove the following auxiliary result.

\begin{lemma}\label{l5}
Let $0<a\le a_0$, where $a_0$ is defined in \eqref{11d}, and $u_1$,
$u_2$ from the segment $[-a,a]$ satisfy \eqref{13d}.

$1)$ If $u_1u_2\ge 0$, then
\begin{equation}\label{21d}
(u_1-u_2)^2\le
0.235309(a^2-u_1u_2),
\end{equation}
\begin{equation}\label{22d}
a^2-u_1u_2\ge
0.933029 a^2.
\end{equation}

$2)$ If $u_1u_2< 0$, then
\begin{equation}\label{23d}
(u_1-u_2)^2\le 2(a^2-u_1u_2), \quad a^2-u_1u_2\ge a^2.
\end{equation}
\end{lemma}

\begin{proof}
1) Let $u_1u_2\ge 0$.
From \eqref{13d} it follows that
$$
u_1^2+2u_1u_2+ u_2^2=(u_1+u_2)^2\le (a^2+u_1u_2)^2=a^4+2a^2u_1u_2+u_1^2u_2^2,
$$
therefore, taking into account non-negativeness of $u_1u_2$, we have
$$
(u_1-u_2)^2=(u_1)^2+(u_2)^2-2u_1u_2\le a^4-2(2-a^2)u_1u_2+u_1^2u_2^2\le a^4-2a^2u_1u_2+u_1^2u_2^2=
(a^2-u_1u_2)^2.
$$
From this it follows that
\begin{equation*} (u_1-u_2)^2\le
a^2(a^2-u_1u_2)\le  a_0^2(a^2-u_1u_2)=0.235309(a^2-u_1u_2)
\end{equation*}
and  \eqref{21d} is proved.

Now we will estimate the value of $a^2-u_1u_2$ from below. Taking
into account that $(1-u_1)(1-u_2)\ge 1-a^2$, we obtain
$$
u_1u_2=(1-u_1)(1-u_2)-(1-u_1)-(1-u_2)+1\le (1-u_1)(1-u_2)-2\sqrt{(1-u_1)(1-u_2)}+1$$$$=(1-\sqrt{(1-u_1)(1-u_2)})^2\le(1-\sqrt{1-a^2})^2.
$$
Thus, \begin{multline*}
a^2-u_1u_2\ge a^2-(1-\sqrt{1-a^2})^2= a^2-\frac{a^4}{(1+\sqrt{1-a^2})^2}\\
=\left(1-\frac{a^2}{(1+\sqrt{1-a^2})^2}\right)a^2\ge\left(1-\frac{a_0^2}{(1+\sqrt{1-a_0^2})^2}\right)a^2=0.933029 a^2,
\end{multline*}
i.e., \eqref{22d} is proved.

2) Let now $u_1u_2< 0$. Then,
$$
(u_1-u_2)^2=(u_1)^2+(u_2)^2-2u_1u_2\le 2(a^2-u_1u_2),
$$
and the first inequality from \eqref{23d} is proved, and the second one is evident.
\end{proof}

Now we will turn to the proof of \eqref{19d}.

Consider two cases.

1) Let
$u_1u_2\ge 0$. Then, with the help of Lemma~\ref{l5} we have $\alpha=0.235309$, $\beta=0.933029,$ and $\gamma=0.449368$. Therefore, \eqref{20d} holds.

2) Let
$u_1u_2<0$. Then, from Lemma~\ref{l5} we have $\alpha=2$, $\beta=1,$ and $\gamma=0.343805$. Thus, in this case \eqref{20d} is also valid.

\section*{Conflict of interest}
The authors declare no conflict of interest.

\end{document}